\documentclass[pdflatex,sn-mathphys-num]{sn-jnl}

\usepackage{graphicx}%
\usepackage{multirow}%
\usepackage{amsmath,amssymb,amsfonts}%
\usepackage{amsthm}%
\usepackage{mathrsfs}%
\usepackage[title]{appendix}%
\usepackage{xcolor}%
\usepackage{textcomp}%
\usepackage{manyfoot}%
\usepackage{booktabs}%
\usepackage{algorithm}%
\usepackage{algorithmicx}%
\usepackage{algpseudocode}%
\usepackage{listings}%
\usepackage{mathtools}
\usepackage{subfig}
\usepackage{enumitem}

\theoremstyle{thmstyleone}%

\theoremstyle{thmstyletwo}%

\theoremstyle{thmstylethree}%
\newtheorem{definition}{Definition}%
\newtheorem{theorem}{Theorem}[section]
\newtheorem{proposition}{Proposition}[section]

\newtheorem{lemma}{Lemma}[theorem]

\DeclareMathOperator{\sign}{sign}
\DeclareMathOperator{\E}{\mathbb{E}}

\DeclareMathOperator{\cond}{\left.\right|}

\begin{document}

\title[Mixed Poisson distributions]{Mixed Poisson families with real-valued mixing distributions}

\author*[1]{\fnm{F. William} \sur{Townes}}\email{ftownes@andrew.cmu.edu}

\affil*[1]{\orgdiv{Department of Statistics and Data Science}, \orgname{Carnegie Mellon University}, \orgaddress{\street{5000 Forbes Ave}, \city{Pittsburgh}, \postcode{15213}, \state{PA}, \country{USA}}}

\abstract{Mixed Poisson distributions provide a flexible approach to the analysis of count data with overdispersion, zero inflation, or heavy tails. Since the Poisson mean must be nonnegative, the mixing distribution is typically assumed to have nonnegative support. We show this assumption is unnecessary and real-valued mixing distributions are also possible. Informally, the mixing distribution merely needs to have a light (subexponential) left tail and a small amount of probability mass on negative values. We provide several concrete examples, including the mixed Poisson- extreme stable family, where the mixing distribution has a power law tail.}




\maketitle

\section{Introduction}\label{sec1}

Mixed Poisson families \cite{johnsonUnivariateDiscreteDistributions2005,grandellMixedPoissonProcesses2020} are widely used to model count data with overdispersion, zero inflation, or heavy tails in a variety of applications including finance \cite{wuenscheUsingMixedPoisson2007,fariaFinancialDataModeling2013,heaukulaniModellingFinancialVolume2024}, biology \cite{el-shaarawiModellingSpeciesAbundance2011,loveModeratedEstimationFold2014,townesQuantileNormalizationSinglecell2020}, and the physical sciences \cite{kitauraRecoveringNonlinearDensity2010,tarnowskyFirstStudyNegative2013,taroniHowManyStrong2023}. Suppose $X\sim F$ is a random variable and $[Y\cond X]$ follows the Poisson distribution with $\E[Y\cond X] = X$. We say that (marginally) $Y$ has a mixed Poisson distribution after integrating out $X$. A prominent example is the negative binomial distribution, which is a mixed Poisson where $X$ is gamma distributed. It seems natural to assume that $X$ must be restricted to have nonnegative support. However, this is in fact unnecessary; the Hermite distribution is obtained by allowing $X\sim\mathcal{N}(\mu,\sigma^2)$, so long as $\mu\geq\sigma^2$ \cite{kempAlternativeDerivationHermite1966}. 

Here, we identify general conditions under which real-valued mixing distributions can be used to produce mixed Poisson distributions. As a motivating example, consider the stable distributions \cite{nolan:2018}, which are the limits of rescaled sums of potentially heavy-tailed random variables. The stable family has two key parameters: $\alpha\in (0,2]$ controls the heaviness of the tails and $\beta\in[-1,1]$ controls the skew. When $\beta=1$, the distribution is maximally skewed to the right. In this case, for $\alpha\in (0,1)$ the left tail can be bounded at zero. Such ``positive stable'' distributions can straightforwardly be used to create mixed Poisson distributions with heavy tails. However, for $\beta=1$ and $\alpha \in [1,2]$ the support is $\mathbb{R}$. 
This would seem to disqualify these ``extreme stable'' distributions as candidates for producing Poisson mixtures. But the Gaussian distribution is stable with $\alpha=2$ and \cite{kempAlternativeDerivationHermite1966} showed it can be used to produce a Poisson mixture. We are then left with a curious gap of $\alpha\in [1,2)$ in the possible values for heavy-tailed count modeling. We will show that in fact mixed Poisson extreme stable distributions are well-defined for the full range of $\alpha\in(0,2]$, so long as $\beta=1$ and certain constraints on location and scale are satisfied. Of course, our results encompass non-stable families as well.

\section{Preliminaries}
\subsection{Subweibull random variables and exponential tilting}
\begin{definition}
\label{def:laplace-transform}
The bilateral Laplace-Stieltjes transform (BLT) of a random variable $X$ with distribution function $F$ is
\[\mathcal{L}_X(t)= \E[\exp(-tX)]=\int_{-\infty}^\infty \exp(-tx)dF(x)\]
\end{definition}
We do not restrict $X$ to be nonnegative or to have a density function. In the special case that $\mathcal{L}_X(t)<\infty$ for all $t$ in an open interval around $t=0$, then $X$ has a moment generating function (MGF) which is $M_X(t)= \E[\exp(tX)]=\mathcal{L}_X(-t)$. 

\begin{definition}
\label{def:subweibull}
A random variable $X$ is \textit{q-subweibull} if $\E[\exp(\lambda^q \vert X\vert^q)]<\infty$ for some $\lambda>0$ and $q>0$. $X$ is \textit{strictly q-subweibull} if the condition is satisfied for all $\lambda>0$. 
\end{definition}

We use the term \textit{subexponential} to refer to 1-subweibull random variables (i.e. those having MGFs).


\begin{definition}
\label{def:exp-tilt}
Let $X$ be a random variable with distribution function $F$. If the BLT satisfies $\mathcal{L}_X(-\theta)=\E[\exp(\theta X)]<\infty$ for some $\theta\neq 0$, then the \textit{exponentially tilted distribution} is given by
\[F_\theta(x) = \int_{-\infty}^x \frac{\exp(\theta t)}{\mathcal{L}_X(-\theta)}dF(t)\]
\end{definition}

\begin{proposition}
\label{prop:exp-tilt}
Preservation of subweibull tails under exponential tilting. Let $\theta$ be any real number.
\begin{enumerate}
 \item If $X\sim F$ is q-subweibull ($q>1$), then the exponentially tilted variable $Z\sim F_\theta$ is also q-subweibull.
 \item If $X\sim F$ is strictly q-subweibull ($q\geq1$), the exponentially tilted variable $Z\sim F_\theta$ is also strictly q-subweibull.
 \item If $X\sim F$ is not q-subweibull ($q>1$), then $Z\sim F_\theta$ is also not q-subweibull.
\end{enumerate}
\end{proposition}

The proof of Proposition \ref{prop:exp-tilt}, along with further discussion of exponential tilting and subweibull properties, is provided in \cite{townesExponentialTiltingSubweibull2024}. See also \cite{kuchibhotlaMovingSubGaussianityHighdimensional2022,vladimirovaSubWeibullDistributionsGeneralizing2020}.

\subsection{Mixed Poisson distributions}

\begin{definition}
\label{def:valid-pmf}
A function $f(n)$ on the domain $\mathbb{N}_0$ is a valid probability mass function (PMF) if it satisfies
\begin{enumerate}[label=(\roman*)]
 \item $f(n)\geq 0$ for all $n\in\mathbb{N}_0$
 \item $\sum_{n=0}^\infty f(n) = 1$. 
\end{enumerate}
\end{definition}


\begin{definition}
\label{def:pgf}
The probability generating function (PGF) of a discrete nonnegative random variable $X$ is given by
\[G_X(z)=\E[z^X] = \sum_{n=0}^\infty z^n\Pr(X=n)\]
\end{definition}

We will need to verify whether a given function is a PGF for some random variable. 

\begin{lemma}
\label{lem:valid-pgf}
An analytic function $G(z)$ is a valid PGF if and only if it satisfies
\begin{enumerate}[label=(\roman*)]
 \item $G(1)=1$
 \item $G(z)$ is continuous for $z\in[0,1]$
 \item Absolute monotonicity: finite derivatives $G^{(k)}(z)\geq 0$ for all $z\in (0,1)$
\end{enumerate}
\end{lemma}
The proof is provided on p. 223 of \cite{fellerIntroductionProbabilityTheory1971}. 
We use notation from \cite{gurlandInterrelationsCompoundGeneralized1957} to concisely describe mixed Poisson distributions.

\begin{definition}
\label{def:mixed}
Let $X$ be a random variable with distribution $F$. If random variable $Y$ has conditional distribution $[Y\cond X]\sim Poi(X)$, then the marginal distribution of $Y$ is the \textit{mixed Poisson distribution} generated by $F$, denoted with
\[Y\sim \left(Poi \bigwedge F\right)\]
and having PMF
\begin{equation}
\label{eq:mpoi-pmf-simple}
f(n)=\frac{\E\left[X^n e^{-X}\right]}{n!}
\end{equation}
provided $f(n)$ is a valid PMF.
\end{definition}

\begin{lemma}
\label{lem:pgf-mpoi}
The PGF of a mixed Poisson random variable $Y\sim \left(Poi\bigwedge F\right)$ is
\[G_Y(z) = \mathcal{L}_X(1-z)\]
where $\mathcal{L}_X(t)$ is the BLT of the mixing distribution $F$, provided $G_Y(z)$ is a valid PGF.
\end{lemma}
\begin{proof}
\begin{align*}
G_Y(z) = \E\big[\E[z^Y\cond X]\big] &= \E\left[\sum_{n=0}^\infty \frac{z^n X^n \exp(-X)}{n!}\right]\\
&= \E\left[\exp(-X)\sum_{n=0}^\infty \frac{(zX)^n}{n!}\right]= \E\left[\exp(-X+zX)\right]= \mathcal{L}_X(1-z)
\end{align*}
\end{proof}




\section{General conditions for real-valued mixing distributions}

We will first give conditions based on the probability mass function (PMF). 

\subsection{Sufficient conditions}
We will refer to $f(n)$ as in Equation \ref{eq:mpoi-pmf-simple} as a \textit{candidate PMF} for some candidate mixing variable $X\sim F$ under consideration. The mixture $Poi\bigwedge F$ exists if $f(n)$ is a valid PMF.
It is well known that $\Pr(X<0)=0$ is a sufficient condition.

\begin{lemma}
\label{lem:nonneg-not-nec}
Let $X\sim F$ be a random variable. Then $\Pr(X<0)=0$ is a sufficient but not necessary condition for the existence of the mixed distribution $Poi \bigwedge F$.
\end{lemma}
\begin{proof}
Sufficiency: the PMF $f(n)\geq 0$ for all $n$ since all terms in Equation \ref{eq:mpoi-pmf-simple} are nonnegative. And
\[\sum_n f(n) = \E\left[\sum_n \frac{X^n e^{-X}}{n!}\right] = 1\]
where the interchange of sum and expectation is justified by Tonelli's Theorem.
Lack of necessity: If $X\sim \mathcal{N}(\mu,\sigma^2)$ with $\mu\geq\sigma^2$, then $Poi \bigwedge F$ is the Hermite distribution \cite{kempAlternativeDerivationHermite1966} with PGF
\[G_Y(z)=\exp\left[a_1 (z-1) + a_2 (z^2-1)\right]\]
where $a_2=\sigma^2/2$ and $a_1=\mu-\sigma^2$. Since $\Pr(X<0)>0$, this serves as a counterexample.
\end{proof}

This motivates us to consider less restrictive sufficient conditions that still hold when $\Pr(X<0)>0$.

\begin{theorem}
\label{thm:mpoi-pmf-suff}
Let $X$ be a random variable with distribution $F$. Define nonnegative random variables $A=[-X\cond X<0]$ and $B=[X\cond X\geq 0]$. The following conditions are sufficient to establish the existence of mixed distribution $Poi\bigwedge F$ when $\Pr(X<0)>0$.
\begin{enumerate}[label=(\roman*)]
\item for all odd $n\in\mathbb{N}$
\[\E\left[A^n e^{A}\right]\Pr(X<0)\leq \E\left[B^n e^{-B}\right]\Pr(X\geq 0)\]
\item $\E[\exp(2 A)]<\infty$
\end{enumerate}
The PMF of $Y\sim \left(Poi\bigwedge F\right)$ is then given by
\begin{equation}
\label{eq:mpoi-pmf-complex}
f(n) = \frac{(-1)^n\E\left[A^n e^{A}\right]\Pr(X<0)+\E\left[B^n e^{-B}\right]\Pr(X\geq 0)}{n!}
\end{equation}
\end{theorem}
\begin{proof}
Equation \ref{eq:mpoi-pmf-complex} is clearly equivalent to Equation \ref{eq:mpoi-pmf-simple}. 
\begin{align*}
f(n) &= \int \frac{x^n e^{-x}}{n!} dF(x) = \frac{\E\left[X^n e^{-X}\right]}{n!}\\
&= \frac{1}{n!}\left(\E\left[X^n e^{-X}\cond X<0\right]\Pr(X<0) + \E\left[X^n e^{-X}\cond X\geq0\right]\Pr(X\geq 0)\right)
\end{align*}
Assume condition (i) holds. This implies $f(n)\geq 0$ for all odd $n$. For even $n$, all terms in Equation \ref{eq:mpoi-pmf-complex} are nonnegative, because $A$ is a positive random variable and $B$ is a nonnegative random variable. This proves $f(n)\geq 0$ for all $n\in \mathbb{N}_0$ as required.

It remains to prove that $\sum_n f(n)=1$. Assume (ii) holds. It is clear that 
\[\E\left[\sum_n \frac{X^n \exp(-X)}{n!}\right] = 1\]
even when $X$ takes on negative values. However, we cannot justify the interchange of expectation and summation with Tonelli's Theorem since $X$ takes on negative values. Therefore, we must appeal to Fubini's Theorem, which requires the following sum to be finite.
\begin{align*}
\sum_n &\E\left[\left\vert \frac{X^n \exp(-X)}{n!}\right\vert\right] = \sum_n \frac{1}{n!}\E\left[\vert X\vert^n e^{-X}\right]\\
&= \sum_n \frac{\E\left[\vert X\vert^n e^{-X}\cond X\geq 0\right]\Pr(X\geq 0) + \E\left[\vert X\vert^n e^{-X}\cond X<0\right]\Pr(X<0)}{n!}\\
&= \sum_n \frac{\E\left[B^n e^{-B}\right]\Pr(X\geq 0) + \E\left[A^n e^{A}\right]\Pr(X<0)}{n!}\\
&= \E\left[\sum_n \frac{B^n \exp(-B)}{n!}\right] + \E\left[\sum_n \frac{A^n \exp(A)}{n!}\right]\\
&= (1) + \E[\exp(2 A)]
\end{align*}
Here the interchange of sum and expectation is justified by Tonelli's Theorem since everything is nonnegative. By the assumption of condition (ii) the expectation in the last line is finite and Fubini's Theorem justifies the conclusion that
\[\sum_n f(n) = \sum_n \E\left[\frac{X^n\exp(-X)}{n!}\right] = \E\left[\sum_n \frac{X^n\exp(-X)}{n!}\right] = 1\]
\end{proof}

\subsection{Necessary conditions}

Theorem \ref{thm:mpoi-pmf-suff} broadens the landscape of mixed Poisson distributions. To give an idea of how large the space could be, we circumscribe it by providing necessary conditions. Furthermore, in many cases it may be more straightforward to check these conditions than those of Theorem \ref{thm:mpoi-pmf-suff}(i), so as to speedily rule out candidate distributions that are invalid.

\begin{theorem}
\label{thm:mpoi-pmf-nec}
Let $X$ be a random variable with distribution $F$. Define nonnegative random variables $A=[-X\cond X<0]$ and $B=[X\cond X\geq 0]$. The following conditions are necessary for the mixed distribution $Poi\bigwedge F$ to exist.
\begin{enumerate}
  \item Either $\Pr(X<0)=0$ or $\Pr(X>0)>0$.
  \item If $\Pr(X<0)>0$ then $A$ must be subexponential with $\E[e^{A}]<\infty$.
  \item If $\Pr(X<0)>0$ and $B$ is q-subweibull, then $A$ must also be q-subweibull.
\end{enumerate}
\end{theorem}
\begin{proof}
$(1)$: Suppose $\Pr(X<0)>0$ and $\Pr(X>0)=0$. Then for $n$ odd, 
\[\Pr(Y=n) = \frac{-1}{n!}\E\left[A^n e^{A}\right]\Pr(X<0) + (0)\Pr(X=0)< 0\]
since $A >0$, which violates Definition \ref{def:valid-pmf}(i).

$(2)$: Suppose $A$ is not subexponential. Since $A$ is bounded on the left by zero, we have for all $t\leq 0$,
\[\E[e^{tA}] = \int_{0}^\infty e^{tx} dF^-(x)\leq \int_{0}^\infty (1) dF^-(x) = 1< \infty\]
Therefore it must be that $\E[e^{tA}] = \infty$ for all $t>0$.
\begin{align*}
\Pr(Y=0) &\geq \frac{1}{0!}\E\left[A^0 e^{A}\right]\Pr(X<0)\\
&\geq \E\left[e^{(1)A}\right]\Pr(X<0) = \infty
\end{align*} 
which violates Definition \ref{def:valid-pmf}(ii). Therefore $A$ must be subexponential. By a similar argument if $A$ is subexponential but $\E[e^{(1)A}]=\infty$ then $\Pr(Y=0)=\infty$ which again violates Definition \ref{def:valid-pmf}(ii).

$(3)$: We will show that if $\Pr(X<0)>0$ and $B$ is q-subweibull but $A$ is not, then there exists an odd integer $n$ for which $\Pr(Y=n)$ cannot have a nonnegative value. We have already established in $(2)$ that $A$ must be subexponential, so choose $q>1$. The condition that $\Pr(Y=n)\geq 0$ is equivalent to, for odd $n$:
\begin{align*}
0<\frac{\Pr(X<0)}{\Pr(X\geq 0)} &\leq  \frac{\E\left[B^n e^{-B}\right]}{\E\left[A^n e^{A}\right]}
\end{align*}
Let $F^-$ and $F^+$ denote the distributions of $A$ and $B$, respectively. Since $B$ is nonnegative, the exponentially tilted distribution $F^+_\theta$ exists for all $\theta\leq 0$. Therefore we can define $V\sim F^+_{(-1)}$ so that
\begin{align*}
\E\left[B^n e^{-B}\right] &= \mathcal{L}_{B}(1)\int x^n \frac{\exp(-x)}{\mathcal{L}_{B}(1)}dF^+(x)\\ 
&= K_B\int z^n dF^+_{(-1)}(z)\\ 
&= K_B\E[V^n]
\end{align*}
where we have defined the constant $K_B = \mathcal{L}_{B}(1) = \E[\exp(-B)]<\infty$ for notational convenience.
By Proposition \ref{prop:exp-tilt}, $V$ is also q-subweibull which implies there is some $K_1>0$ such that for all $n\geq 1$,
\begin{align*}
\E[V^n]^{1/n}&\leq K_1 n^{1/q}\\
\E\left[B^n e^{-B}\right] = K_B\E[V^n] &\leq K_B K_1^n n^{n/q}
\end{align*}

Since by assumption $A$ is subexponential with $K_A = M_{A}(1) = \E[\exp(A)]<\infty$, we may define $U$ as a random variable following the exponentially tilted distribution $F^-_{(1)}$ so that 
\begin{align*}
\E\left[A^n e^{A}\right] &= M_{A}(1)\int x^n \frac{\exp(x)}{M_{A}(1)}dF^-(x)\\ 
&= K_A \int z^n dF^-_{(1)}(z)\\ 
&= K_A \E[U^n]
\end{align*}
By Proposition \ref{prop:exp-tilt}, $U$ is not q-subweibull which implies for all $K>0$ there exists some $n\geq 1$ such that $\E[U^n]^{1/n}> K n^{1/q}$. Equivalently,
\[\limsup_{n\to\infty} \frac{\E[U^n]^{1/n}}{n^{1/q}} = \infty \]
This establishes that there must exist an infinite subsequence $n_k$ satisfying $\lim_{k\to\infty} n_k=\infty$ such that
\[\lim_{k\to\infty} \frac{m(n_k)}{n_k^{1/q}} = \infty \]
where $m(n)=\E[U^n]^{1/n}$ is a nondecreasing function of $n$ (it is just the $L_p$ norm sequence). Let $h(n)=m(n)/n^{1/q}$. Then
\[h(n_k+1) = \frac{m(n_k+1)}{(n_k+1)^{1/q}} \geq \frac{m(n_k)}{(n_k+1)^{1/q}} = \frac{h(n_k)n_k^{1/q}}{(n_k+1)^{1/q}}\]
Since $\lim_{k\to\infty} h_{nk} = \infty$ and
\[\lim_{k\to\infty} \frac{n_k^{1/q}}{(n_k+1)^{1/q}} = 1\]
This implies $\lim_{k\to\infty} h(n_k+1)=\infty$. Now, there are two possibilities. If $\{h(n_k)\}$ contains an infinite number of odd terms, this directly provides a subsequence of odd terms converging to infinity as well. If instead $\{h(n_k)\}$ contains only a finite number of odd terms, then it contains an infinite number of even terms. Therefore we can find an infinite number of odd terms in the subsequence $l_k = n_k+1$ such that $\lim_{k\to\infty} h(l_k)=\infty$. Together this implies
\[\limsup_{n~odd;\\~n\to\infty} \frac{\E[U^n]^{1/n}}{n^{1/q}} = \infty\]
Furthermore, since $f(x)=Kx$ is a continuous, increasing function for any $K>0$, 
\[\limsup_{n~odd;\\~n\to\infty} \frac{\E[U^n]^{1/n}}{K n^{1/q}} = \infty\]
Finally, since there must be an infinite number of terms greater than one, and $f(x)=a^x$ is continuous and increasing,
\[\limsup_{n~odd;\\~n\to\infty} \left(\frac{\E[U^n]^{1/n}}{K n^{1/q}}\right)^n = \infty\]
Combining the results for the $B$ and $A$ terms produces the desired contradiction.
\begin{align*}
\limsup_{n~odd;~n\to\infty} \frac{\E\left[A^n e^{A}\right]}{\E\left[B^n e^{-B}\right]} &\geq \limsup_{n~odd;~n\to\infty}\frac{\E[U^n] K_A}{K_1^n n^{n/q} K_B} = \infty\\
\liminf_{n~odd;~n\to\infty} \frac{\E\left[B^n e^{-B}\right]}{\E\left[A^n e^{A}\right]} &= 0
\end{align*}
This means that for any value of $\Pr(X<0)>0$ we can produce an odd $n$ such that $\Pr(Y=n)$ cannot have a nonnegative value. Since this violates Definition \ref{def:valid-pmf}(i), we conclude that if $B$ is q-subweibull, then $A$ must also be q-subweibull. 
\end{proof}

Although we have established that the left tail of the candidate distribution must be subexponential, and no heavier than the right tail (in a subweibull sense), we have left open the possibility that the two tails could be of a similar order, for example if they both decay at an exponential but not strictly subexponential rate. This is further explored in a later section. First we detail necessary and sufficient conditions based on the BLT of the mixing measure.

\subsection{Conditions based on the bilateral Laplace-Stieltjes transform}

The class of valid mixing distributions can be characterized by the following properties of the BLT:
\begin{proposition}
\label{prop:valid-blt}
Let $X$ be a random variable with distribution function $F$. The mixed Poisson distribution $Poi\bigwedge F$ exists iff $X$ has a BLT $\mathcal{L}_X(t)$ that is completely monotone for $t\in[0,1]$, i.e.:
\begin{enumerate}[label=(\roman*)]
 \item $\mathcal{L}_X(t)$ is continuous for $t\in[0,1]$
 \item For all $k\in \mathbb{N}_0$, finite derivatives satisfy $(-1)^k\mathcal{L}_X^{(k)}(t)\geq 0$ for $t\in (0,1)$.
\end{enumerate}
\end{proposition} 
\begin{proof}
First, assume $\mathcal{L}_X(t)$ satisfies the stated properties. We will show that $G(z)=\mathcal{L}_X(1-z)$ is a valid PGF. First, that $\mathcal{L}_X(t)$ is a BLT implies $\mathcal{L}_X(0)=1$ hence $G(1)=1$. Continuity of $\mathcal{L}_X(t)$ on $t\in[0,1]$ implies continuity of $G(z)$ on $z\in[0,1]$. Complete monotonicity of $\mathcal{L}_X(t)$ on $t\in(0,1)$ implies absolute monotonicity of $G(z)$ on $z\in (0,1)$. The converse is equally straightforward: if $G(z)$ is a valid PGF of a mixed Poisson distribution, then the mixing distribution must have a BLT with the stated properties.
\end{proof}
Bernstein's theorem\cite{bernsteinFonctionsAbsolumentMonotones1929,widderLaplaceTransform2010} holds that the class of completely monotone functions with nonnegative domain is equivalent to the class of one-sided Laplace-Stieltjes transforms (LST) of finite measures with nonnegative support. In other words, every nonnegative random variable has a LST that is completely monotone on $[0,\infty)$. Proposition \ref{prop:valid-blt} shows that the class of valid mixing distributions is considerably broader, permitting real-valued mixing distributions. 

\section{Examples of real-valued mixing distributions}
\subsection{Two point mixture}
The two point distribution is a mixture of a positive atom (point mass) at $b$, a negative atom at $-a$, and probabilities $\Pr(X=-a)=p$ and $\Pr(X=b)=(1-p)$ where $a,b>0$. The Rademacher distribution is a special case. Let $\phi=p/(1-p)$ be the odds of the negative atom.
\begin{proposition}
\label{prop:mpoi-2pt}
The mixed Poisson- two point distribution with negative atom at $-a$, positive atom at $b$, and odds of negative atom $\phi$ exists if $b\geq a$ and $(b/a)\exp(-a-b)\geq \phi$.
The PMF is given by
\[f(n)=(1-p)\frac{b^n \exp(-b)}{n!} + p\frac{(-a)^n \exp(a)}{n!}\]
\end{proposition}
\begin{proof}
We will show that the provided constraints are sufficient to satisfy the requirements of Theorem \ref{thm:mpoi-pmf-suff}. Condition $(ii)$ is trivially satisfied since $\E[\exp(2 A)] = \exp(2a)<\infty$. For condition $(i)$, we must have for all odd $n\in\mathbb{N}$,
\begin{align*}
(1-p)b^n e^{-b}- p a^n e^a &\geq 0\\
(b/a)^n e^{-a-b}&\geq \frac{p}{1-p}=\phi
\end{align*}
Let $h(n)=(b/a)^n\exp(-a-b)$. By assumption $h(1)\geq\phi$. Now assume $h(n)\geq \phi$. We have $h(n+1)=(b/a)h(n)\geq h(n)\geq \phi$ since $b\geq a$. Therefore by induction condition $(i)$ is satisfied for all $n\in\mathbb{N}$.
\end{proof}
Proposition \ref{prop:mpoi-2pt} shows that the mixing distribution of a mixed Poisson may include a negative atom arbitrarily far from zero (with suitably small probability mass). Alternatively, the mass of the negative component can be arbitrarily close to one if the atom is close enough to zero. The PMF can exhibit stark multimodality (Figure \ref{fig:mpoi-2pt-pmf}).

\begin{figure}[tb]
\centering
\subfloat[Mixing distribution]{
  \includegraphics[width=.49\linewidth]{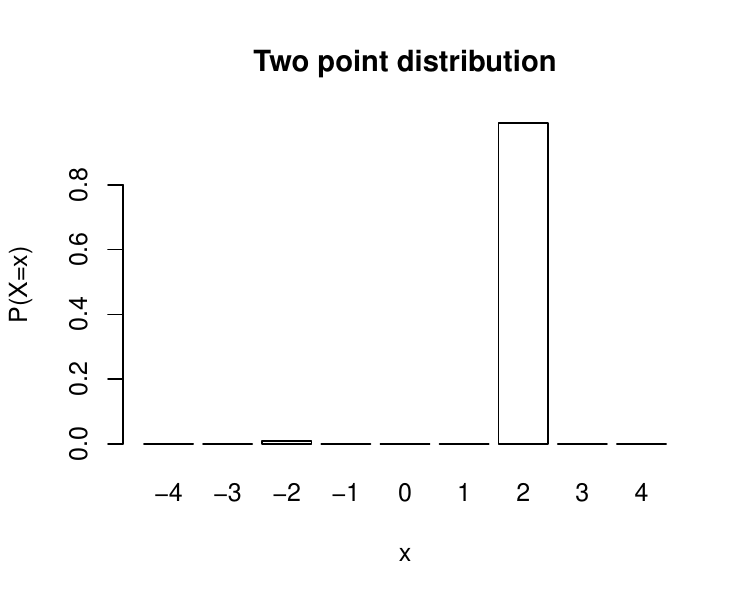}
}
\subfloat[Mixed Poisson distribution]{
  \includegraphics[width=.49\linewidth]{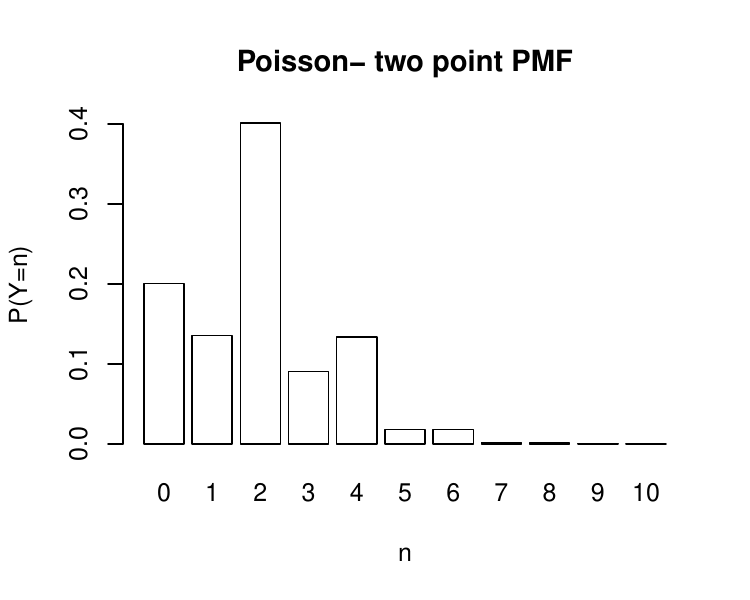}
}
\caption{\small Mixed Poisson- two point probability mass function with positive atom $b=2$, negative atom $a=-2$, and probability of negative atom $.009$.}
\label{fig:mpoi-2pt-pmf}
\end{figure}

\subsection{Asymmetric Laplace mixture}
The asymmetric Laplace distribution has $A = [-X\cond X<0]$ distributed as $Exp(\lambda_1)$, $B=[X\cond X\geq 0]$ distributed as $Exp(\lambda_2)$, and $\Pr(X<0)=p$. We refer to the rates $\lambda_1$ and $\lambda_2$ as the left and right tail parameters, respectively. Let $\phi=p/(1-p)$ be the odds of the negative tail.
\begin{proposition}
\label{prop:mpoi-asym-laplace}
The mixed Poisson- asymmetric Laplace distribution with left and right tail parameters $\lambda_1,\lambda_2$ and odds of negative tail $\phi$ exists if $\lambda_1\geq \lambda_2+2$ and 
\[\phi\leq \frac{\lambda_2}{\lambda_1}\left(\frac{\lambda_1-1}{\lambda_2+1}\right)^2\]
The PMF is given by
\[f(n)=(1-p)\frac{\lambda_2}{(\lambda_2+1)^{n+1}} + p(-1)^n\frac{\lambda_1}{(\lambda_1-1)^{n+1}}\]
\end{proposition}
\begin{proof}
We will show that the provided constraints are sufficient to satisfy the requirements of Theorem \ref{thm:mpoi-pmf-suff}. Condition $(ii)$ is satisfied since $\E[\exp(2A)]<\infty$ for $\lambda_1>2$ and by assumption $\lambda_1\geq \lambda_2+2 > 2$. Note that
\[\E\left[A^n e^{A}\right] = \lambda_1\int_0^\infty x^n\exp(-(\lambda_1-1)x)dx= \lambda_1 \frac{\Gamma(n+1)}{(\lambda_1-1)^{n+1}}\]
and
\[\E\left[B^n e^{-B}\right]=\lambda_2 \int_0^\infty x^n\exp(-(\lambda_2+1)x)dx = \lambda_2\frac{\Gamma(n+1)}{(\lambda_2+1)^{n+1}} \]
so condition $(i)$ is equivalent to
\begin{align*}
\lambda_2\frac{\Gamma(n+1)}{(\lambda_2+1)^{n+1}}(1-p)&\geq\lambda_1 \frac{\Gamma(n+1)}{(\lambda_1-1)^{n+1}}p\\
\frac{\lambda_2}{\lambda_1}\left(\frac{\lambda_1-1}{\lambda_2+1}\right)^{n+1}&\geq\frac{p}{1-p}=\phi
\end{align*}
By assumption this is true for $n=1$. Assuming it is true for $n$ we can see this implies the case of $n+1$ whenever $(\lambda_1-1)/(\lambda_2+1)\geq 1$, which is implied by the assumption of $\lambda_1\geq \lambda_2+2$. By induction this implies the inequality is satisfied for all $n\in\mathbb{N}$. 
\end{proof}
Proposition \ref{prop:mpoi-asym-laplace} shows that the mixing distribution of a mixed Poisson may include support on all negative numbers, with a left tail decaying at an exponential rate, so long as the right tail decays sufficiently slowly and the odds of the negative tail are small enough. Alternatively, the mass of the negative component can be arbitrarily close to one if the left tail decays sufficiently rapidly ($\lambda_1$ is large). The PMF of this family can also be multimodal (Figure \ref{fig:mpoi-asym-laplace-pmf}).

\begin{figure}[tb]
\centering
\subfloat[Mixing distribution]{
  \includegraphics[width=.49\linewidth]{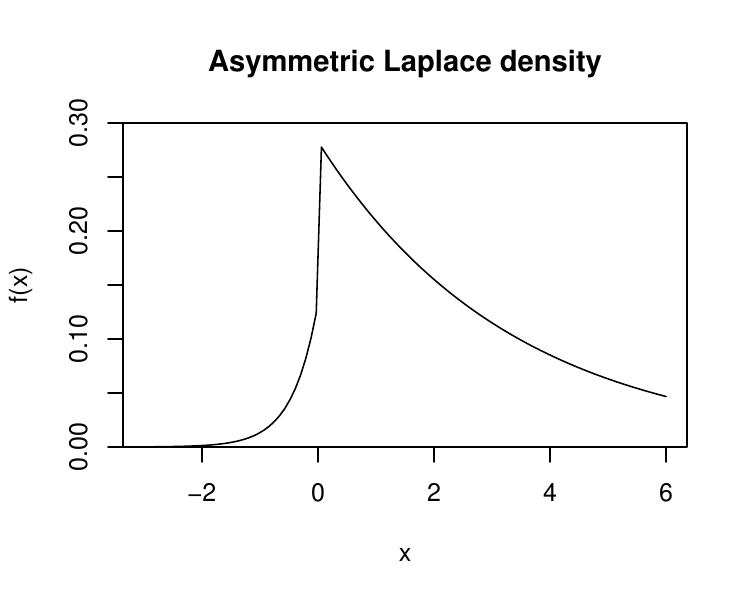}
}
\subfloat[Mixed Poisson distribution]{
  \includegraphics[width=.49\linewidth]{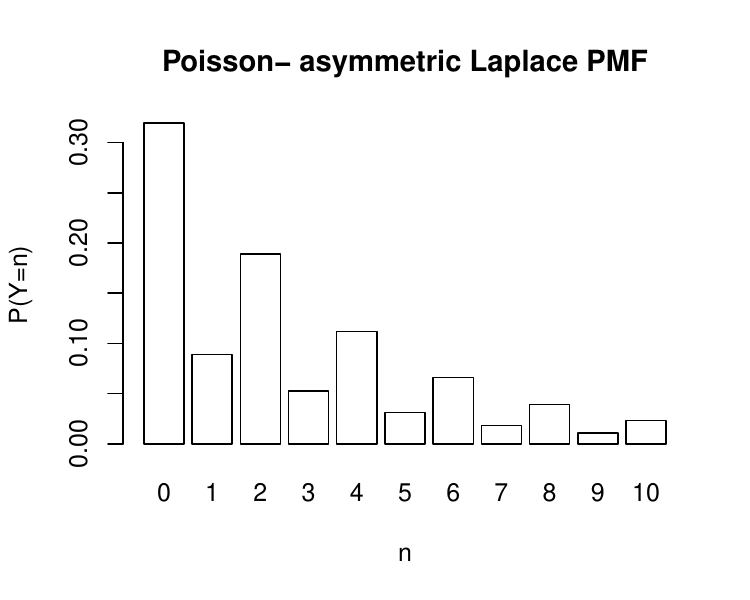}
}
\caption{\small Mixed Poisson- asymmetric Laplace probability mass function with left tail rate $\lambda_1=2.3$, right tail rate $\lambda_2=0.3$ and probability of negative tail $p=.058$.}
\label{fig:mpoi-asym-laplace-pmf}
\end{figure}

\subsection{Poisson-extreme stable}
\subsubsection{Stable distributions}
The central limit theorem shows that normalized sums of random variables with finite variance converge to the Gaussian distribution. In the case of infinite variance, the limiting distribution is instead a stable distribution \cite{nolan:2018,samorodnitskyStableNonGaussianRandom1994}. 
\begin{definition}\cite{nolan:2018}
\label{def:stable-cf}
A random variable $X\sim S_\alpha(\sigma,\beta,\delta)$ is stable if and only if it has the following characteristic function:
\begin{equation*}
  \E\left[e^{itX}\right] =
    \begin{cases}
      \exp\left[it\delta - \sigma^\alpha\vert t \vert^\alpha\left(1-i\beta\sign(t)\tan\frac{\alpha\pi}{2}\right)\right] & \alpha\neq 1\\
      \exp\left[it\delta - \sigma\vert t \vert \left(1+i\beta\frac{2}{\pi}\sign(t)\log\vert t\vert\right)\right] & \alpha=1.
    \end{cases}
\end{equation*}
where the parameters are $\delta\in \mathbb{R}$ for location, $\beta\in[-1,1]$ for skewness, and $\alpha\in (0,2]$ as the index parameter. The scale parameter is constrained to $\sigma\geq 0$ for $\alpha=1$ and $\sigma>0$ otherwise. $X$ is \textit{strictly stable} if and only if either $\alpha\neq 1$ and $\delta=0$ or $\alpha=1$ and $\sigma\beta=0$.
\end{definition}
The constraint on $\sigma$ ensures that degenerate distributions with $\sigma=0$ are considered to have $\alpha=1$. The case of $\alpha=2$ is the Gaussian family. Smaller values of $\alpha$ lead to heavier tails. A nondegenerate stable random variable with $\alpha<2$ has infinite variance, and with $\alpha\leq 1$ has an infinite mean.

The BLT is not finite for most non-Gaussian stable distributions because they have heavy tails on both sides. However, in the special case of \textit{extreme stable} distributions having $\beta=1$ (maximally skewed to the right), the left tail becomes strictly subexponential \cite{nolan:2018} and the BLT is given \cite{guptaMultiscalingPropertiesSpatial1990,samorodnitskyStableNonGaussianRandom1994} by
\begin{equation}
\label{eq:stable-laplace-transform}
  \mathcal{L}_X(t) = \E[\exp(-tX)] = 
    \begin{cases}
      \exp\left(-t\delta-\sec\left(\frac{\pi\alpha}{2}\right)\sigma^\alpha t^\alpha\right) & \alpha\neq 1\\
      \exp\left(-t\delta+\sigma\frac{2}{\pi}t\log t \right) & \alpha=1.
    \end{cases}
\end{equation}
which is finite whenever $t\geq 0$. Note that $\sec\frac{\pi\alpha}{2}$ is positive if $0<\alpha<1$ and negative if $1<\alpha\leq 2$. 

\subsubsection{Mixed Poisson-extreme stable}

All nondegenerate stable distributions are absolutely continuous \cite{nolan:2018}. It seems natural to consider whether stable distributions can be used to form Poisson mixtures for modeling discrete data. When $\alpha<1$ this is straightforward because the extreme stable family has support only on the positive reals (it is often called the \textit{positive stable} family). However, for $\alpha \in [1,2]$, the extreme stable family has support on the entire real line. 
Consider the function $G_Y(z) = \mathcal{L}_X(1-z)$. If this function satisfies the requirements of a PGF, then $Y$ is a mixed Poisson distribution with a skewed stable mixing distribution. 
\begin{theorem} \textit{Mixed Poisson-stable family}
\label{thm:mpoi-stable}
Let $X$ be a stable random variable with BLT as in Equation \ref{eq:stable-laplace-transform}. If the location parameter $\delta$ satisfies the following constraint, then $X$ can be used as the mixing parameter in a mixed Poisson distribution. 
\begin{equation}
\label{eq:delta-lim}
\delta \geq 
\begin{cases}
-\alpha \sec\left(\frac{\pi\alpha}{2}\right)\sigma^\alpha & \alpha\neq 1\\
\sigma\frac{2}{\pi} & \alpha=1
\end{cases}
\end{equation}
The PGF of the resulting mixed Poisson family is
\begin{equation}
\label{eq:mpoi-stable-pgf}
G(z)=
  \begin{cases}
    \exp\left((z-1)\delta-\sec\left(\frac{\pi\alpha}{2}\right)\sigma^\alpha(1-z)^\alpha\right) & \alpha\neq 1\\
    \exp\left((z-1)\delta+\sigma\frac{2}{\pi}(1-z)\log(1-z)\right) & \alpha=1.
  \end{cases}
\end{equation}
\end{theorem}
\begin{proof}
We must show that $G(z)$ has the properties of Lemma \ref{lem:valid-pgf}. It is obvious that $G(1)=1$, $G(z)$ is continuous for $z\in [0,1]$, and $G(0)\geq 0$ for all $\alpha\in(0,2]$. It remains to show that the derivatives $G^{(k)}(z)$ are finite and nonnegative for all $k\in\mathbb{N}$ and $z\in (0,1)$. The first derivative is
\[G'(z) = G(z)r(z)\]
where 
\begin{align*}
r(z) = 
\begin{cases}
  \delta + \alpha \sec\left(\frac{\pi\alpha}{2}\right)\sigma^\alpha(1-z)^{\alpha-1} & \alpha\neq 1\\
  \delta+\sigma\frac{2}{\pi}\left(-1-\log(1-z)\right) & \alpha=1.
\end{cases}
\end{align*}
The constraints of Equation \ref{eq:delta-lim} imply $r(0)\geq 0$, which implies $G'(0)\geq 0$. Since $r(z)$ is an increasing function for $z\in [0,1)$, this implies $r(z)\geq r(0)\geq 0$ for all $z\in (0,1)$ as required.
We will now show that all derivatives of $r(z)$ are nonnegative for $z\in(0,1)$, separately for $\alpha\neq 1$ and $\alpha=1$.

\textbf{Case of} $\mathbf{\alpha\neq 1}$. 
The remaining derivatives of $r(z)$ are
\begin{align*}
r'(z) &= \alpha (\alpha-1) \sec\left(\frac{\pi\alpha}{2}\right)\sigma^\alpha(1-z)^{\alpha-2}(-1)\\
r^{(k)}(z) &= \alpha (\alpha-1)\ldots (\alpha-k)\sec\left(\frac{\pi\alpha}{2}\right)\sigma^\alpha(1-z)^{\alpha-k-1}(-1)^k\\
&= \alpha(1-\alpha)\ldots(k-\alpha)\sec\left(\frac{\pi\alpha}{2}\right)\sigma^\alpha(1-z)^{\alpha-k-1}
\end{align*}
It is clear that $r^{(k)}(z)$ is finite for all $z<1$. We will show that $r^{(k)}(0)\geq 0$ by induction. For $\alpha<1$, all terms are positive. For $\alpha>1$ the negative $(1-\alpha)$ term cancels with the negative secant term. Therefore $r'(0)\geq 0$ for all $\alpha\neq 1$. Now assume $r^{(k)}(0)\geq 0$ for some $k\geq 1$. We have $r^{(k+1)}(z) = r^{(k)}(z)(k+1-\alpha)(1-z)^{-1}$ and $r^{(k+1)}(0) = r^{(k)}(0)(k+1-\alpha)(1)$. All terms are positive. Therefore $r^{(k)}(0)\geq 0$ for all $k\in\mathbb{N}$. Finally, since $r^{(k)}(z)$ is increasing in $z$ for $z\in (0,1)$, we have that $r^{(k)}(z)\geq r^{(k)}(0)\geq 0$ as well.

\textbf{Case of} $\mathbf{\alpha= 1}$. 
The remaining derivatives of $r(z)$ are
\begin{align*}
r'(z) &= \sigma\frac{2}{\pi}(1-z)^{-1}\\
r^{(k)}(z) &= \sigma\frac{2}{\pi}(1-z)^{-k}(k-1)!
\end{align*}
These are all finite for $z<1$ and $r^{(k)}(0) = \sigma\frac{2}{\pi}(k-1)!\geq 0$ for all $k\in\mathbb{N}$. Again, $r^{(k)}(z)\geq r^{(k)}(0)\geq 0$ for all $z\in (0,1)$.

Returning to the original function $G(z)$ we have (via the general Leibniz rule),
\[G^{(k+1)}(z) = \sum_{j=0}^{k} \binom{k}{j} G^{(k-j)}(z)r^{(j)}(z) \]
We have already shown that $G(z)$ and $G'(z)$ are finite for $z<1$. If $G^{(k)}(z)$ is finite on the same range, then so is $G^{(k+1)}(z)$ since it is a finite combination of the lower derivatives and the $r^{(k)}(z)$ derivatives, which are all finite. By induction this implies $G^{(k)}(z)$ is finite for all $k\in\mathbb{N}$. As for nonnegativity, we have already shown $G(z)\geq 0$ and $G'(z)$ is nonnegative for $z\in (0,1)$ as long as the condition in Equation \ref{eq:delta-lim} is satisfied. Assume $G^{(k)}(z)\geq 0$. Then $G^{(k+1)}(z)$ is also since it consists of a finite combination of products of nonnegative terms. 
This shows $G(z)$ satisfies all the requirements of Lemma \ref{lem:valid-pgf}, therefore it is a valid PGF.
\end{proof}

\subsection{Numerical evaluation of Poisson-stable PMF}

The Poisson- extreme stable PMF can be obtained by repeated differentiation of the PGF. Fortunately, the derivatives have a relatively simple form. We will use a slightly different parameterization for convenience:
\[\gamma = \begin{cases}
-\sec\left(\frac{\alpha \pi}{2}\right)\sigma^\alpha & \alpha\neq 1\\
\sigma\frac{2}{\pi} & \alpha=1
\end{cases}\]
The constraints of Theorem \ref{thm:mpoi-stable} simplify to $\gamma\geq 0$ for $\alpha\geq 1$, $\gamma<0$ for $\alpha<1$, and $\delta\geq \alpha\gamma$. Define the ``R-function'' as
\begin{equation}
r(z) = \frac{d}{dz}\log G(z) = 
\begin{cases}
  \delta - \alpha \gamma(1-z)^{\alpha-1} & \alpha\neq 1\\
  \delta-\gamma\left(1+\log(1-z)\right) & \alpha=1.
\end{cases}
\end{equation}
Clearly $r(0)=\delta-\alpha\gamma$ for all $\alpha\in (0,2]$. For $\alpha\in (0,1]$ its derivatives ($k\geq 1$) at zero can be expressed as
\begin{equation}
\label{eq:rfunc-alpha01}
r^{(k)}(0) = 
\begin{cases}
-\alpha\gamma \frac{\Gamma(k-\alpha+1)}{\Gamma(1-\alpha)} & \alpha\in (0,1)\\
\gamma \Gamma(k) & \alpha=1\\
\alpha(\alpha-1)\gamma\frac{\Gamma(k+1-\alpha)}{\Gamma(2-\alpha)} & \alpha\in(1,2)
\end{cases}
\end{equation}
For $\alpha=2$ (Hermite distribution), $r'(0)=2\gamma$ and $r^{(k)}(0)=0$ for all $k\geq 2$. The PMF can be evaluated for $\alpha\in (0,2]$ by the following recursion:
\begin{align*}
\Pr(Y=0)&= \exp\left(-\delta + \mathbf{1}_{\alpha\neq 1}\gamma\right)\\
\Pr(Y=n)&= \frac{1}{n}\sum_{j=0}^{n-1}\frac{\Pr(Y=j)}{(n-1-j)!}r^{(n-1-j)}(0)
\end{align*}

We provide a numerically stable implementation in R as part of the companion code repository at \url{https://github.com/willtownes/mixed-poisson}. Since each term of the PMF requires evaluation of all preceding terms, the computational complexity is $\mathcal{O}(n^2)$.

\section{Discussion}
We have shown that the mixing distribution of a mixed Poisson family need not be restricted to have nonnegative support. Informally, the mixing distribution merely needs to have a light left tail and a small amount of probability mass on negative values. 

In the context of hierarchical models, the mixing distribution may be interpreted as a Bayesian prior. It is well known that one may use ``improper'' distributions as priors so long as the resulting posterior is proper. However, to our knowledge it has not been proposed to use proper distributions whose support is not matched to that of the parameter of interest. For example, consider the mixed binomial distribution. If the total count $N$ is fixed, the constraints on the mixing distribution for the success probability $p$ apply only to the first $N$ moments. One could therefore specify a mixed Binomial- Gaussian distribution for $N=2$. It would be interesting to explore how this perspective relates to other efforts to relax the assumption of a hierarchical model as a complete specification of the data generating process, such as robust Bayes \cite{millerRobustBayesianInference2019} and Bayesian estimating equations \cite{bissiriGeneralFrameworkUpdating2016}. 

In the field of stochastic processes, an overdispersed Poisson process can be produced by rescaling the time index using a subordinator. For example, the negative binomial process may be produced by subjecting a Poisson process to a gamma process subordinator \cite{kozubowskiDistributionalPropertiesNegative2009}. The subordinator is classically restricted to be a nonnegative process \cite{kallenbergFoundationsModernProbability2021}. Since the subordinator resembles a process analog of the mixing distribution, this raises the question of whether general, real-valued subordinators could be considered as well.

While the distributional families explored here are intriguing, they are not yet ready to be used for practical statistical applications. In future work we plan to develop numerical algorithms for evaluating mixed Poisson likelihoods, sampling, estimation, and inference. In the meantime we hope our results will inspire further research into flexible approaches to the specification of mixed distributions and analogous stochastic processes.

\backmatter

\bmhead{Acknowledgements}

The authors would like to thank Gennady Gorin, Arun Kumar Kuchibhotla, Siyuan Ma, Zack McCaw, Gonzalo Mena, Kimberly Sellers, Justin Silverman, Weijing Tang, Valerie Ventura, Larry Wasserman, Nathan Welch, Matt Werenski, and Mike Wu for helpful comments and suggestions.

\section*{Declarations}

\begin{itemize}
\item Funding- Not applicable
\item Conflict of interest/Competing interests- Not applicable
\item Ethics approval and consent to participate- Not applicable
\item Consent for publication- Not applicable
\item Data availability- Not applicable
\item Materials availability- Not applicable
\item Code availability- \url{https://github.com/willtownes/mixed-poisson}
\item Author contribution- everything by single author
\end{itemize}


\bibliography{references}

\end{document}